\documentclass[12pt,reqno]{amsart}
\usepackage{etoolbox}

   \makeatletter

 \patchcmd{\@setaddresses}{\scshape\ignorespaces}{\ignorespaces}{}{} 

\appto\maketitle{%
\let\@makefnmark\relax  \let\@thefnmark\relax
\ifx\@empty\addresses\else\@footnotetext{%
  \vskip-\bigskipamount\@setaddresses}
  }
\def\enddoc@text{}
\makeatother

\makeatletter
\patchcmd\maketitle
  {\uppercasenonmath\shorttitle}
  {}
  {}{}
\patchcmd\maketitle
  {\@nx\MakeUppercase{\the\toks@}}
  {\the\toks@}
  {}
  {}{}
\patchcmd\@settitle{\uppercasenonmath\@title}{\Large}{}{}
\patchcmd\@setauthors
  {\MakeUppercase{\authors}}
  {\authors}
  {}{}
\makeatother
\usepackage{amsmath,amssymb,amsthm}
\usepackage{color}
\usepackage{url}
\usepackage{tikz-cd}
\usepackage[utf8]{inputenc}
\usepackage[T1]{fontenc}
\newtheorem{theorem}{Theorem}[section]

\newtheorem{corollary}{Corollary}[section]

\newtheorem{lemma}{Lemma}[section]
\newtheorem{remark}{Remark}[section]

\numberwithin{equation}{section}
\usepackage[colorlinks=true]{hyperref}
\hypersetup{urlcolor=blue, citecolor=red , linkcolor= blue}
\usepackage[capitalise,noabbrev,nameinlink]{cleveref}      
\begin{document}
\address{$^{[1]}$ University of Sfax, Sfax, Tunisia.}
\email{\url{kais.feki@hotmail.com}}


\date{\today}
\author[Kais Feki] {\Large{Kais Feki}$^{1}$}
\title[Inequalities for the $A$-joint numerical radius of two operators and their applications]{Inequalities for the $A$-joint numerical radius of two operators and their applications}

\maketitle

\begin{abstract}
Let $\big(\mathcal{H}, \langle \cdot\mid \cdot\rangle \big)$ be a complex Hilbert space and $A$ be a positive (semidefinite) bounded linear operator on $\mathcal{H}$. The semi-inner product induced by $A$ is given by ${\langle x\mid y\rangle}_A := \langle Ax\mid y\rangle$, $x, y\in\mathcal{H}$ and defines a seminorm ${\|\cdot\|}_A$ on $\mathcal{H}$. This makes $\mathcal{H}$ into a semi-Hilbert space. The $A$-joint numerical radius of two $A$-bounded operators $T$ and $S$ is given by
\begin{align*}
\omega_{A,\text{e}}(T,S) = \sup_{\|x\|_A= 1}\sqrt{\big|{\langle Tx\mid x\rangle}_A\big|^2+\big|{\langle Sx\mid x\rangle}_A\big|^2}.
\end{align*}
In this paper, we aim to prove several bounds involving $\omega_{A,\text{e}}(T,S)$. Moreover, several inequalities related to the $A$-Davis-Wielandt radius of semi-Hilbert space operators is established. Some of the obtained bounds generalize and refine some earlier results of Zamani and Shebrawi [Mediterr. J. Math. 17, 25 (2020)].
\end{abstract}

\section{Introduction and Preliminaries}\label{s1}
Let $\mathcal{B}(\mathcal{H})$ denote the $C^*$-algebra of all bounded linear operators acting on a complex Hilbert space
$\mathcal{H}$ with an inner product $\langle\cdot\mid\cdot\rangle$ and the corresponding norm $\|\cdot\|$. Throughout this paper, by an operator we mean a bounded linear operator. Let $T^*$ denote the adjoint of an operator $T$. Further, the range and the kernel of $T$ are denoted by $\mathcal{R}(T)$ and $\mathcal{N}(T)$, respectively. In addition, the cone of all positive operators on  $\mathcal{H}$ is given by
$$\mathcal{B}(\mathcal{H})^+:=\left\{A\in \mathcal{B}(\mathcal{H})\,;\,\langle Ax\mid x\rangle\geq 0,\;\forall\;x\in \mathcal{H}\;\right\}.$$
Any $A\in \mathcal{B}(\mathcal{H})^+$ induces the following semi-inner product:
$$\langle\cdot\mid\cdot\rangle_{A}:\mathcal{H}\times \mathcal{H}\longrightarrow\mathbb{C},\;(x,y)\longmapsto\langle x\mid y\rangle_{A} :=\langle Ax\mid y\rangle.$$
Observe that the seminorm induced by  $\langle\cdot\mid\cdot\rangle_{A}$ is given by $\|x\|_A=\langle x\mid x\rangle_A^{1/2}$, for every $x\in \mathcal{H}$. This makes $\mathcal{H}$ into a semi-Hilbert space. It is not difficult to verify that $\|\cdot\|_A$ is a norm on $\mathcal{H}$ if and only if $A$ is injective, and that $(\mathcal{H},\|\cdot\|_A)$ is complete if and only if $\mathcal{R}(A)$ is a closed subspace of $\mathcal{H}$. From now on, we suppose that $A\in\mathcal{B}(\mathcal{H})$ is always a positive (nonzero) operator and we denote the $A$-unit sphere of $\mathcal{H}$ by $\mathbb{S}^A(0,1)$, that is,
$$\mathbb{S}^A(0,1):=\{x\in \mathcal{H}\,;\; \|x\|_A=1\}.$$
For $T\in\mathcal{B}(\mathcal{H})$, the $A$-numerical radius and the $A$-Crawford number of $T$ are given by
\begin{align*}
\omega_A(T) = \sup\Big\{\big|{\langle Tx\mid x\rangle}_A\big|\,; \,\,\, x\in\mathbb{S}^A(0,1)\Big\}
\end{align*}
and
\begin{align*}
c_A(T) = \inf \big\{|{\langle Tx\mid x\rangle}_A|\,;\,\,\, x\in\mathbb{S}^A(0,1)\big\},
\end{align*}
respectively (see \cite{saddi,bakfeki01,zamani1} and the references therein). It should be emphasized here that it may happen that $\omega_A(T) = + \infty$ for some $T\in\mathcal{B}(\mathcal{H})$ (see \cite{feki02}).

Let $T \in \mathcal{B}(\mathcal{H})$. An operator $S\in\mathcal{B}(\mathcal{H})$ is called an $A$-adjoint of $T$ if for every $x,y\in \mathcal{H}$, the identity $\langle Tx\mid y\rangle_A=\langle x\mid Sy\rangle_A$ holds (see \cite{acg1}). So, $S$ is an $A$-adjoint of $T$ if and only if $S$ is solution in $\mathcal{B}(\mathcal{H})$ of the equation $AX=T^*A$. This kind of equations can be studied by using Douglas theorem \cite{doug} which says that the operator equation $TX=S$ has a solution $X\in \mathcal{B}(\mathcal{H})$ if and only if $\mathcal{R}(S) \subseteq \mathcal{R}(T)$ which in turn equivalent to the existence of a positive number $\lambda$ such that $\|S^*x\|\leq \lambda \|T^*x\|$ for all $x\in \mathcal{H}$. In addition, among its many solutions it has only one, denoted by $Q$, which satisfies $\mathcal{R}(Q) \subseteq \overline{\mathcal{R}(T^{*})}$. Such $Q$ is said the reduced solution of the equation $TX=S$. Obviously, the existence of an $A$-adjoint operator is not guaranteed. The subspace of all operators admitting $A$-adjoints is denoted by $\mathcal{B}_{A}(\mathcal{H})$. By Douglas theorem, it holds that
$$\mathcal{B}_{A}(\mathcal{H})=\left\{T\in \mathcal{B}(\mathcal{H})\,;\;\mathcal{R}(T^{*}A)\subset \mathcal{R}(A)\right\}.$$
Let $T\in \mathcal{B}_A(\mathcal{H})$. The reduced solution of the operator equation $AX=T^*A$ is denoted by $T^{\sharp_A}$. Moreover we have, $T^{\sharp_A}=A^\dag T^*A$. Here $A^\dag$ denotes the Moore-Penrose inverse of $A$ (see \cite{acg2}). From now on, for simplicity we will write $X^{\sharp}$ instead of $X^{\sharp_A}$ for every $X \in \mathcal{B}_A({\mathcal{H}})$. Notice that if $T \in \mathcal{B}_A({\mathcal{H}})$, then $T^{\sharp} \in \mathcal{B}_A({\mathcal{H}})$, $(T^{\sharp})^{\sharp}=P_{\overline{\mathcal{R}(A)}}TP_{\overline{\mathcal{R}(A)}}$ and $((T^{\sharp})^{\sharp})^{\sharp}=T$. Here $P_{\overline{\mathcal{R}(A)}}$ denotes the orthogonal projection onto $\overline {\mathcal{R}(A)}$. Further, if $S\in \mathcal{B}_A(\mathcal{H})$ then $TS \in\mathcal{B}_A({\mathcal{H}})$ and $(TS)^{\sharp}=S^{\sharp}T^{\sharp}.$ For an account of results concerning $T^{\sharp}$, we refer the reader to \cite{acg1,acg2}. Again, an application of Douglas theorem gives
$$\mathcal{B}_{A^{1/2}}(\mathcal{H})=\left\{T \in \mathcal{B}(\mathcal{H})\,;\;\exists \,\lambda > 0\,;\;\|Tx\|_{A} \leq \lambda \|x\|_{A},\;\forall\,x\in \mathcal{H}  \right\}.$$
It $T\in\mathcal{B}_{A^{1/2}}(\mathcal{H})$, then $T$ is called $A$-bounded. Notice that $\mathcal{B}_{A}(\mathcal{H})\subseteq\mathcal{B}_{A^{1/2}}(\mathcal{H})$ (see \cite{acg3,feki01}). The seminorm of an operator $T\in\mathcal{B}_{A^{1/2}}(\mathcal{H})$ is given by
\begin{equation}\label{semii}
\|T\|_A:=\sup_{\substack{x\in \overline{\mathcal{R}(A)},\\ x\not=0}}\frac{\|Tx\|_A}{\|x\|_A}=\sup\big\{{\|Tx\|}_A\,; \,\,x\in \mathbb{S}^A(0,1)\big\}<\infty.
\end{equation}
Notice that the second equality in \eqref{semii} has been proved in \cite{fg}. We mention here that $\|\cdot\|_A$ and $\omega_A(\cdot)$ are equivalent seminorms on $\mathcal{B}_{A^{1/2}}(\mathcal{H})$. More precisely, for every $T\in \mathcal{B}_{A^{1/2}}(\mathcal{H})$, we have
\begin{equation}\label{refine1}
\tfrac{1}{2} \|T\|_A\leq\omega_A(T) \leq \|T\|_A,
\end{equation}
(see \cite{bakfeki01}). Further, it was shown in \cite{bakfeki01} that
\begin{equation}\label{apower}
\omega_A(T^n)\leq [\omega_A(T)]^n,
\end{equation}
for every $T\in\mathcal{B}_{A^{1/2}}(\mathcal{H})$ and all positive integer $n$. Before, we move on it is crucial to recall that for every $T,S\in \mathcal{B}_{A^{1/2}}(\mathcal{H})$ we have
\begin{equation}\label{crucial0}
\|TS\|_A\leq \|T\|_A\cdot\|S\|_A,
\end{equation}
(see \cite{bakfeki01}). Recall that an operator $T\in\mathcal{B}(\mathcal{H})$ is said to be $A$-selfadjoint if $AT$ is selfadjoint. Observe that if $T$ is $A$-selfadjoint, then $T\in\mathcal{B}_A(\mathcal{H})$. It was shown in \cite{feki01} that for every $A$-selfadjoint operator $T$ we have
\begin{equation}\label{aself1}
\|T\|_{A}=\omega_A(T).
\end{equation}
Further, an operator $T$ is called $A$-positive if $AT\geq0$ and we write $T\geq_{A}0$. Obviously, an $A$-positive operator is $A$-selfadjoint since $\mathcal{H}$ is a complex Hilbert space. It can be checked that $T^{\sharp} T\geq_{A}0$ and $TT^{\sharp}\geq_{A}0$. Moreover, for every $T\in\mathcal{B}_A(\mathcal{H})$ we have
\begin{align}\label{diez}
{\|T^{\sharp} T\|}_A = {\|TT^{\sharp}\|}_A = {\|T\|}^2_A = {\|T^{\sharp}\|}^2_A,
\end{align}
(see \cite[Proposition 2.3.]{acg2}). Now, an operator $T\in\mathcal{B}_A(\mathcal{H})$ is called $A$-normal if $TT^{\sharp} = T^{\sharp}T$ (see \cite{bakfeki04}). It is obvious that every selfadjoint operator is normal. However, an $A$-selfadjoint operator is
not necessarily $A$-normal (see \cite[Example 5.1]{bakfeki04}).

The $A$-joint numerical radius of a $d$-tuple of operators $(T_1,\ldots,T_d)\in \mathcal{B}(\mathcal{H})^d:= \mathcal{B}(\mathcal{H})\times \cdots \times \mathcal{B}(\mathcal{H})$ was defined in \cite{bakfeki01} by
\begin{equation*}
\omega_{A,\text{e}}(T_1,\ldots,T_d)=\displaystyle\sup\left\{\left(\displaystyle\sum_{k=1}^d|\langle T_kx\mid x\rangle_A|^2\right)^{\frac12};\;x\in \mathbb{S}^A(0,1)\right\}.
\end{equation*}
Notice that the particular case $d=1$ is the $A$-numerical radius of an operator $T$ which recently attracted the attention of several mathematicians (see, e.g., \cite{bakfeki01,bakfeki04,bfeki,feki01,feki02,zamani2,zamani1,zamani3} and the references therein). Some interesting properties of $A$-joint numerical radius of $A$-bounded operators were given in \cite{bakfeki01}. In particular, it is established that for an operator tuple $(T_1,\ldots,T_d)\in \mathcal{B}_A(\mathcal{H})^d$ we have
\begin{align} \label{1.1}
\frac{1}{2\sqrt{d}}\left\|\sum_{k=1}^{d}T_k^{\sharp}T_k\right\|^\frac12\leq \omega_{A,\text{e}}(T_1,\ldots,T_d)\leq \left\|\sum_{k=1}^{d}T_k^{\sharp}T_k\right\|^\frac12.
\end{align}
By using \eqref{1.1}, the present author proved recently in \cite{feki02} that for every $T\in\mathcal{B}_A(\mathcal{H})$ we have
  \begin{equation}\label{16}
  \frac{1}{16}\|T^{\sharp} T+TT^{\sharp}\|_A\le  \omega_A^2\left(T\right) \le \frac{1}{2}\|T^{\sharp} T+TT^{\sharp}\|_A.
  \end{equation}

Recently, the $A$-Davis-Wielandt radius of an operator $T\in\mathcal{B}(\mathcal{H})$ is defined by K. Feki et al in \cite{fekisidha2019} by
\begin{align*}
d\omega_A(T)
&:=\displaystyle\sup\left\{\sqrt{|\langle Tx\mid x\rangle_A|^2+\|Tx\|_A^4}\,;\;x\in \mathbb{S}^A(0,1)\right\}.
\end{align*}
Notice that it was shown in \cite{fekisidha2019}, that $d\omega_A(T)$ may be equal to $+\infty$ for some $T\in\mathcal{B}(\mathcal{H})$. However, if $T\in\mathcal{B}_{A^{1/2}}(\mathcal{H})$, then we have
$$\max\{\omega_A(T),\|T\|_A^2\}\leq d\omega_A(T) \leq \sqrt{\omega_A(T)^2+\|T\|_A^4}<\infty.$$
Clearly, if $T\in\mathcal{B}_A(\mathcal{H})$, then the $A$-Davis-Wielandt radius can be seen as the $A$-joint numerical radius of the operator tuple $(T,T^\sharp T)$. That is, for $T\in\mathcal{B}(\mathcal{H})$, it holds
\begin{equation}\label{davis}
d\omega_A(T)=\omega_{A,\text{e}}(T,T^\sharp T).
\end{equation}
In this paper we establish several inequalities concerning the $A$-joint numerical radius of two semi-Hilbert space operators. In particular, some related results connecting the $A$-joint numerical radius and the classical $A$-numerical radius are also presented. Moreover, we prove several inequalities involving the $A$-Davis-Wielandt radius and the $A$-numerical radii of $A$-bounded operators. Some of the obtained results cover and extend the work of Drogomir \cite{dra1} and the recent paper of Zamani et al. \cite{zamanicheb2020}.

\section{Results}
In this section, we present our result. In order to establish our first upper bound for the $A$-joint numerical radius of two semi-Hilbert space operators we need the following lemmas.
\begin{lemma}(\cite[Section 2]{acg1})\label{l01}
Let $T\in\mathcal{B}(\mathcal{H})$ be an $A$-selfadjoint operator. Then, $T = T^{\sharp}$ if and only if $T$ is $A$-selfadjoint and $\mathcal{R}(T) \subseteq \overline{\mathcal{R}(A)}$.
\end{lemma}

\begin{lemma}\label{le.2.7}
For every $a, b, c\in \mathcal{H}$
\begin{align}\label{extencu0}
|\langle a\mid b\rangle_A|^2 + |\langle a\mid c\rangle_A|^2
\leq \|a\|_A^2\sqrt{|\langle b\mid b\rangle_A|^2 + 2|\langle b\mid c\rangle_A|^2
+ |\langle c\mid c\rangle_A|^2}.
\end{align}
\end{lemma}
\begin{proof}
Notice first that, by \cite[p. 148]{D.2}, we have
\begin{align}\label{zzkk22}
|\langle x\mid y\rangle|^2 + |\langle x\mid z\rangle|^2
\leq \|x\|^2\Big(|\langle y\mid y\rangle|^2 + 2|\langle y\mid z\rangle|^2
+ |\langle z\mid z\rangle|^2\Big)^{\frac{1}{2}},
\end{align}
for any $x, y, z\in \mathcal{H}$. Now, let $a, b, c\in \mathcal{H}$. It follows, from \eqref{zzkk22}, that
\begin{align*}
&|\langle a\mid b\rangle_A|^2 + |\langle a\mid c\rangle_A|^2\\
& = |\langle A^{1/2}a\mid A^{1/2}b\rangle|^2 + |\langle A^{1/2}a\mid A^{1/2}c\rangle|^2\\
 &\leq \|A^{1/2}a\|^2\sqrt{|\langle A^{1/2}b\mid A^{1/2}b\rangle|^2 + 2|\langle A^{1/2}b\mid A^{1/2}c\rangle|^2
+ |\langle A^{1/2}c\mid A^{1/2}c\rangle|^2}.
\end{align*}
This proves \eqref{extencu0} as desired.
\end{proof}

Our first result in this paper reads as follows.
\begin{theorem}\label{th.2.9}
Let $T,S\in \mathcal{B}_A(\mathcal{H})$. Then,
\begin{align}
\omega_{A,\text{e}}(T,S) \leq \sqrt{\left\|T\right\|_A^4+\left\|S\right\|_A^4+2\omega_A^2(S^\sharp T)}\leq \left\|T\right\|_A^2+\left\|S\right\|_A^2.
\end{align}
\end{theorem}
\begin{proof}
Let $x\in \mathbb{S}^A(0,1)$. By choosing in Lemma \ref{le.2.7} $a = x, b = Tx$ and $c = Sx$ we see that
\begin{align}\label{95}
&\left(|\langle Tx\mid x\rangle_A|^2 + |\langle Sx\mid x\rangle_A|^2\right)^2\nonumber\\
& = \left(|\langle x\mid  Tx\rangle_A|^2 + |\langle x\mid  Sx\rangle_A|^2\right)^2\nonumber\\
&\leq \|x\|_A^4\left(|\langle Tx\mid Tx\rangle_A|^2 + 2|\langle Tx\mid Sx\rangle_A|^2 + |\langle Sx\mid Sx\rangle_A|^2\right)\nonumber\\
& = |\langle T^\sharp Tx\mid x\rangle_A|^2+ |\langle S^\sharp Sx\mid x\rangle_A|^2 + 2|\langle S^\sharp Tx\mid
x\rangle_A|^2 \nonumber\\
&\leq \omega_{A,\text{e}}^2(T^\sharp T,S^\sharp S)+2\omega_A^2(S^\sharp T)\nonumber\\
&\leq \left\|(T^\sharp T)^\sharp T^\sharp T+(S^\sharp S)^\sharp S^\sharp S\right\|_A+2\omega_A^2(S^\sharp T),
\end{align}
where the last inequality follows from the second inequality in \eqref{1.1}. Now, since $T^\sharp T$ is $A$-selfadjoint and satisfies $\mathcal{R}(T^\sharp T) \subseteq \overline{\mathcal{R}(A)}$, then by Lemma \ref{l01} we have $(T^\sharp T)^\sharp=T^\sharp T$. Similarly, $(S^\sharp S)^\sharp=S^\sharp S$. So, by \eqref{95}, we have
$$
\left(|\langle Tx\mid x\rangle_A|^2 + |\langle Sx\mid x\rangle_A|^2\right)^2\leq\left\|(T^\sharp T)^2+(S^\sharp S)^2\right\|_A+2\omega_A^2(S^\sharp T).
$$
By taking the supremum over all $x\in \mathbb{S}^A(0,1)$ in the above inequality we get
\begin{equation}
\omega_{A,\text{e}}(T,S) \leq \sqrt{\left\|(T^\sharp T)^2+(S^\sharp S)^2\right\|_A+2\omega_A^2(S^\sharp T)}.
\end{equation}
Moreover, by using the triangle inequality together with \eqref{crucial0} we obtain
\begin{align*}
\omega_{A,\text{e}}(T,S)
&\leq \sqrt{\left\|T^\sharp T\right\|_A^2+\left\|S^\sharp S\right\|_A^2+2\omega_A^2(S^\sharp T)} \\
 &=\sqrt{\left\|T\right\|_A^4+\left\|S\right\|_A^4+2\omega_A^2(S^\sharp T)}\quad (\text{by }\,\eqref{diez})\\
  &\leq\sqrt{\left\|T\right\|_A^4+\left\|S\right\|_A^4+2\|S^\sharp T\|_A^2}\quad (\text{by }\,\eqref{refine1})\\
   &\leq\sqrt{\left\|T\right\|_A^4+\left\|S\right\|_A^4+2\|S^\sharp \|_A^2\|T\|_A^2}\quad (\text{by }\,\eqref{crucial0})\\
   &\sqrt{\left(\left\|T\right\|_A^2+\left\|S\right\|_A^2\right)^2}=\left\|T\right\|_A^2+\left\|S\right\|_A^2.
\end{align*}
This proves the desired result.
\end{proof}
In what follows, we need the following lemmas.
\begin{lemma}\label{le.2.8}(\cite[Lemma 2.9.]{zamanicheb2020})
For any $z_1,z_2\in \mathbb{C}$, we have
\begin{align*}
\sup\left\{\Big|\alpha z_1 + \beta z_2\Big|^2;\;(\alpha,\beta)\in \mathbb{C}^2,\;|\alpha|^2 + |\beta|^2 \leq 1 \right\} = |z_1|^2 + |z_2|^2.
\end{align*}
\end{lemma}
\begin{lemma}\label{le.2.88}
Let $T, R \in \mathcal{B}_A(\mathcal{H})$. Then, for every $\alpha, \beta\in \mathbb{C}$, we have
\begin{align*}
\|\alpha T + \beta S\|_A^2 \leq (|\alpha|^2 + |\beta|^2)\|T^\sharp T + S^\sharp S\|_A.
\end{align*}
\end{lemma}
\begin{proof}
Let $x\in \mathbb{S}^A(0,1)$. Then, by applying the Cauchy-Schwarz inequality, we see that
\begin{align*}
\|\alpha Tx + \beta Sx\|_A^2
&=\|\alpha A^{1/2}Tx + \beta A^{1/2}Sx\|^2\\
&\leq (|\alpha|^2 + |\beta|^2)(\|A^{1/2}Tx\|^2 + \|A^{1/2}Sx\|^2)\\
&=(|\alpha|^2 + |\beta|^2)(\|Tx\|_A^2 + \|Sx\|_A^2)\\
& = (|\alpha|^2 + |\beta|^2)\big\langle (T^\sharp T + S^\sharp S)x\mid x\big\rangle_A\\
&\leq (|\alpha|^2 + |\beta|^2)\omega_A(T^\sharp T + S^\sharp S)\\
& = (|\alpha|^2 + |\beta|^2)\|T^\sharp T + S^\sharp S\|_A,
\end{align*}
where the last equality follows from \eqref{aself1} since $T^\sharp T + S^\sharp S\geq_A0$. Hence,
$$\|(\alpha T + \beta S)x\|_A^2 \leq (|\alpha|^2 + |\beta|^2)\|T^\sharp T + S^\sharp S\|_A.$$
So, by taking the supremum over all $x\in \mathbb{S}^A(0,1)$ in the above inequality and then using \eqref{semii} we get the desired result.
\end{proof}
Now, we are in a position to prove the following result.
\begin{theorem}\label{th.2.9}
Let $T,S\in \mathcal{B}_A(\mathcal{H})$. Then,
\begin{align}\label{fal9an}
\omega_{A,\text{e}}(T,S) \leq \left[\omega_A\Big((T^\sharp T)^2 + (S^\sharp S)^2\Big) + 2\omega_A^2\big(S^\sharp T\big)\right]^{\frac{1}{4}}.
\end{align}
\end{theorem}
\begin{proof}
Let $x\in \mathbb{S}^A(0,1)$. As in the proof of Theorem \ref{th.2.9}, by choosing in Lemma \ref{le.2.7} $a = x, b = Tx$ and $c = Sx$, we get
\begin{align*}
\left(|\langle Tx\mid x\rangle_A|^2 + |\langle Sx\mid x\rangle_A|^2\right)^2\leq \sup_{x\in \mathbb{S}^A(0,1)}\left( |\langle T^\sharp Tx\mid x\rangle_A|^2+ |\langle S^\sharp Sx\mid x\rangle_A|^2\right) + 2\omega_A^2\big(S^\sharp T).
\end{align*}
Hence, by applying Lemma \ref{le.2.8} we obtain
\begin{align*}
&\left(|\langle Tx\mid x\rangle_A|^2 + |\langle Sx\mid x\rangle_A|^2\right)^2\\
& \leq \sup_{x\in \mathbb{S}^A(0,1)}\left(\sup_{|\alpha|^2 + |\beta|^2 \leq 1}\Big|\alpha\langle T^\sharp Tx\mid x\rangle_A + \beta\langle S^\sharp Sx\mid x\rangle_A\Big|^2\right) + 2\omega_A^2\big(S^\sharp T)\\
& = \sup_{x\in \mathbb{S}^A(0,1)}\Big(\sup_{|\alpha|^2 + |\beta|^2 \leq 1}\Big|\Big\langle \left[\alpha T^\sharp T + \beta S^\sharp S\right]x\mid x\Big\rangle_A\Big|^2\Big) + 2\omega_A^2\big(S^\sharp T)\\
& = \sup_{|\alpha|^2 + |\beta|^2 \leq 1}\Big(\sup_{x\in \mathbb{S}^A(0,1)}\Big|\Big\langle \left[\alpha T^\sharp T + \beta S^\sharp S\right]x\mid x\Big\rangle_A\Big|^2\Big) + 2\omega_A^2\big(S^\sharp T\big).
\end{align*}
On the other hand, it can be see that the operator $\alpha T^\sharp T + \beta S^\sharp S$ is an $A$-selfadjoint operator and then by \eqref{aself1}, we have
$$\sup_{x\in \mathbb{S}^A(0,1)}\Big|\Big\langle \left[\alpha T^\sharp T + \beta S^\sharp S\right]x\mid x\Big\rangle_A\Big|=\|\alpha T^\sharp T + \beta S^\sharp S\|_A.$$
So, by using Lemma \ref{le.2.88}, we get
\begin{align*}
&\left(|\langle Tx\mid x\rangle_A|^2 + |\langle Sx\mid x\rangle_A|^2\right)^2\\
& \leq \sup_{|\alpha|^2 + |\beta|^2 \leq 1}\|\alpha T^\sharp T + \beta S^\sharp S\|_A^2 + 2\omega_A^2\big(S^\sharp T\big)\\
& \leq \sup_{|\alpha|^2 + |\beta|^2 \leq 1}(|\alpha|^2 + |\beta|^2)\Big\|\,(T^\sharp T)^\sharp T^\sharp T + \left[S^\sharp S\right]^\sharp S^\sharp S\,\Big\|_A + 2\omega_A^2\big(S^\sharp T\big)\\
& =\sup_{|\alpha|^2 + |\beta|^2 \leq 1}(|\alpha|^2 + |\beta|^2)\Big\|\,(T^\sharp T)^2 + (S^\sharp S)^2\,\Big\|_A + 2\omega_A^2\big(S^\sharp T\big)\\
& =\Big\|\,(T^\sharp T)^2 + (S^\sharp S)^2\,\Big\|_A + 2\omega_A^2\big(S^\sharp T\big)\\
& =\omega_A\left[(T^\sharp T)^2 + (S^\sharp S)^2\right] + 2\omega_A^2\big(S^\sharp T\big),
\end{align*}
where the last equality follows from \eqref{aself1} since $(T^\sharp T)^2 + (S^\sharp S)^2\geq_A0$. Thus, we get
\begin{align*}
|\langle Tx\mid x\rangle_A|^2 + |\langle Sx\mid x\rangle_A|^2 \leq \sqrt{\omega_A\left[(T^\sharp T)^2 + (S^\sharp S)^2\right] + 2\omega_A^2\big(S^\sharp T\big)},
\end{align*}
for all $x\in \mathbb{S}^A(0,1)$. Finally, by taking the supremum over all $x\in \mathbb{S}^A(0,1)$ in the above inequality we get \eqref{fal9an} as required.
\end{proof}
The following corollary is an immediate consequence of Theorem \ref{th.2.9} and extends \cite[Theorem 2.11]{zamanicheb2020}.
\begin{corollary}\label{cor01}
Let $T\in \mathcal{B}_A(\mathcal{H})$. Then,
\begin{align*}
{d\omega_A}(T) \leq \left[\omega_A\Big((T^\sharp T)^2 + (T^\sharp T)^4\Big)+ 2\omega_A^2\big(T^\sharp T^2\big)\right]^{\frac{1}{4}}.
\end{align*}
\end{corollary}
\begin{proof}
By Lemma \ref{l01}, we have $(T^\sharp T)^\sharp=T^\sharp T$. So, by replacing $S$ by $T^\sharp T$ in \eqref{fal9an} and then using \eqref{davis} we get the required result.
\end{proof}
The following lemma is useful in the sequel.
\begin{lemma}\label{leme1}
For any $a, b, c\in \mathcal{H}$, we have
\begin{align}\label{n03}
|\langle a\mid b\rangle_A|^2 + |\langle a\mid c\rangle_A|^2
\leq \|a\|_A^2\Big(\max\{\|b\|_A^2, \|c\|_A^2\} + |\langle b\mid c\rangle_A|\Big).
\end{align}
\end{lemma}
\begin{proof}
Let $a, b, c\in \mathcal{H}$ be such that $a,b,c\notin \mathcal{N}(A)$. Then, $|\langle a\mid b\rangle_A|^2 + |\langle a\mid c\rangle_A|^2\neq0$. By applying the Cauchy-Schwarz inequality we see that
\begin{align}\label{n01}
\big(|\langle a\mid b\rangle_A|^2 + |\langle a\mid c\rangle_A|^2\big)^2
&=\big(\langle a\mid b\rangle_A\langle b\mid a\rangle_A + \langle a\mid c\rangle_A\langle c\mid a\rangle_A\big)^2\nonumber\\
&=\Big(\langle a\mid \left(\langle a\mid b\rangle_Ab+\langle a\mid c\rangle_Ac\right)\rangle_A\Big)^2\nonumber\\
&= \|a\|_A^2\big\|\langle a\mid b\rangle_Ab+\langle a\mid c\rangle_Ac\big\|_A^2.
\end{align}
On the other hand, one observes
\begin{align}\label{n02}
&\big\|\langle a\mid b\rangle_Ab+\langle a\mid c\rangle_Ac\big\|_A^2\nonumber\\
& = |\langle a\mid b\rangle_A|^2\|b\|_A^2+|\langle a\mid c\rangle_A|^2\|c\|_A^2+2\Re\big(\langle a\mid b\rangle_A \langle c\mid a\rangle_A \langle b\mid c\rangle_A\big)\nonumber\\
&\leq |\langle a\mid b\rangle_A|^2\|b\|_A^2+|\langle a\mid c\rangle_A|^2\|c\|_A^2+2|\langle a\mid b\rangle_A|\cdot| \langle c\mid a\rangle_A| \cdot|\langle b\mid c\rangle_A|\nonumber\\
&\leq |\langle a\mid b\rangle_A|^2\|b\|_A^2+|\langle a\mid c\rangle_A|^2\|c\|_A^2+\left(|\langle a\mid b\rangle_A|^2+| \langle a\mid c\rangle_A|^2\right)|\langle b\mid c\rangle_A|\nonumber\\
&\leq \left(|\langle a\mid b\rangle_A|^2+|\langle a\mid c\rangle_A|^2\right)\Big(\max\{\|b\|_A^2,\|c\|_A^2\}+|\langle b\mid c\rangle_A|\Big).
\end{align}
By combining \eqref{n01} together \eqref{n02} we get \eqref{n03}. If $a,b,c\in \mathcal{N}(A)$, then \eqref{n03} holds trivially. This proves the desired result.
\end{proof}
Now, we are in a position to prove the following theorem.
\begin{theorem}\label{mai9}
Let $T\in \mathcal{B}_A(\mathcal{H})$. Then
\begin{align}\label{derive1}
\omega_{A,\text{e}}(T,S)
&\leq \frac{\sqrt{2}}{2}\sqrt{\left(\|T^\sharp T+S^\sharp S\|_A+\|T^\sharp T-S^\sharp S\|_A \right)+\omega_A(S^\sharp T)} \\
 &\leq \sqrt{2}\sqrt{\max\left(\|T\|_A^2+\|S\|_A^2 \right)+\omega_A(S^\sharp T)}\nonumber.
\end{align}
\end{theorem}
\begin{proof}
Notice first that for any two real numbers $t$ and $s$ we have
\begin{equation}\label{r}
\max\{t,s\}=\frac{1}{2}\left(t+s+|t-s|\right).
\end{equation}
Now, let $x\in \mathbb{S}^A(0,1)$. By letting $a=x$, $b=Tx$ and $c=Sx$ in Lemma \ref{leme1} we get
\begin{align*}
&|\langle Tx\mid x\rangle_A|^2+|\langle Sx\mid x\rangle_A|^2\\
&\leq \max\left\{\|Tx\|_A^2, \|Sx\|_A^2\right\}+|\langle Tx\mid Sx\rangle_A|\\
&=\frac{1}{2}\Big(\|Tx\|_A^2+ \|Sx\|_A^2+\left|\|Tx\|_A^2-\|Sx\|_A^2 \right|\Big)+|\langle Tx\mid Sx\rangle_A|\quad (\text{by }\,\eqref{r})\\
 &=\frac{1}{2}\Big(\langle (T^\sharp T+S^\sharp S)x\mid x\rangle_A+\left|\langle (T^\sharp T-S^\sharp S)x\mid x\rangle_A \right|\Big)+\omega_A(S^\sharp T)\\
  &\leq  \frac{1}{2}\Big(\omega_A(T^\sharp T+S^\sharp S)+\omega_A(T^\sharp T-S^\sharp S)\Big)+\omega_A(S^\sharp T)\\
 &=\frac{1}{2}\Big(\|T^\sharp T+S^\sharp S\|_A+\|T^\sharp T-S^\sharp S\|_A\Big)+\omega_A(S^\sharp T),
\end{align*}
where the last inequality follows from \eqref{aself1} since the operators $T^\sharp T\pm S^\sharp S$ are $A$-selfadjoint. So, we get
$$
|\langle Tx\mid x\rangle_A|^2+|\langle Sx\mid x\rangle_A|^2\leq  \frac{1}{2}\Big(\|T^\sharp T+S^\sharp S\|_A+\|T^\sharp T-S^\sharp S\|_A\Big)+\omega_A(S^\sharp T),$$
for every $x\in \mathbb{S}^A(0,1)$. Thus, by taking the supremum over all $x\in \mathbb{S}^A(0,1)$ in above inequality, we get the first inequality in Theorem \ref{mai9}. Now, the second inequality in Theorem \ref{mai9} follows immediately by applying the triangle inequality and \eqref{diez}.
\end{proof}

We can state the following upper bound for the $A$-Davis-Wielandt radius which generalizes and improves \cite[Theorem 2.14.]{zamanicheb2020}.
\begin{corollary}\label{th.2.12}
Let $T\in \mathcal{B}_A(\mathcal{H})$. Then,
\begin{align*}
{d\omega_A}(T) \leq  \sqrt{\frac{1}{2}\Big[\omega_A\Big((T^\sharp T)^2 + T^\sharp T\Big)+ \omega_A\Big((T^\sharp T)^2 - T^\sharp T\Big)\Big] + \omega_A(T^\sharp T^2)}.
\end{align*}
\end{corollary}
\begin{proof}
Follows immediately by proceeding as in the proof of Corollary \ref{cor01}.
\end{proof}
For the sequel, for any arbitrary operator $T\in {\mathcal B}_A({\mathcal H})$, we write
$$\Re_A(T):=\frac{T+T^{\sharp}}{2}\;\;\text{ and }\;\;\Im_A(T):=\frac{T-T^{\sharp}}{2i}.$$
Furthermore, it is useful to recall the following results.
\begin{lemma}(\cite{feki02})\label{adma}
Let $T\in \mathcal{B}(\mathcal{H})$ be an $A$-selfadjoint operator. Then, $T^{\sharp}$ is $A$-selfadjoint and
\begin{equation*}
({T^{\sharp}})^{\sharp}=T^{\sharp}.
\end{equation*}
\end{lemma}
\begin{lemma}(\cite[Theorem 5.1]{bakfeki04})\label{pos2}
Let $T \in \mathcal{B}(\mathcal{H})$ be an $A$-selfadjoint operator. Then, for any positive integer $n$ we have
\begin{equation*}
\|T^n\|_A=\|T\|_A^n.
\end{equation*}
\end{lemma}
As an application of Theorem \ref{mai9}, we derive the following upper bound of the $A$-numerical radius of operators in $\mathcal{B}_{A}(\mathcal{H})$.
\begin{corollary}\label{twil}
Let $T\in \mathcal{B}_{A}(\mathcal{H})$. Then,
\begin{align}\label{rr01}
\omega_A(T) \le \frac{1}{2}\sqrt{ \|T^\sharp T+TT^\sharp\|_A+\|T^2+(T^\sharp)^2\|_A+ \omega_A\Big((T^\sharp+T)(T-T^\sharp)\Big) }.
\end{align}
Moreover, the inequality \eqref{rr01} is sharp.
\end{corollary}
\begin{proof}
Let $T\in \mathcal{B}_{A}(\mathcal{H})$. Clearly we have $T=\Re_A(T) + i\Im_A(T)$. This implies that $T^{\sharp}=[\Re_A(T)]^{\sharp} - i[\Im_A(T)]^{\sharp}$. Moreover, we see that
\begin{align}\label{newdog}
\omega_A^2(T^{\sharp})
& = \sup\left\{|\langle T^{\sharp}x\mid x\rangle_A|^2\,; \,\, x\in \mathbb{S}^A(0,1)\right\}\nonumber\\
& = \sup\left\{|\langle [\Re_A(T)]^{\sharp}x\mid x\rangle_A|^2+|\langle [\Im_A(T)]^{\sharp}x\mid x\rangle_A|^2\,; \,\,x\in \mathbb{S}^A(0,1)\right\}\nonumber\\
&=\omega_{A,\text{e}}^2\Big([\Re_A(T)]^{\sharp},[\Im_A(T)]^{\sharp}\Big).
\end{align}
Since $\omega_A(T)=\omega_A(T^{\sharp})$, then by using \eqref{newdog}) and applying \eqref{derive1} for $T=[\Re_A(T)]^{\sharp}$ and $S=[\Im_A(T)]^{\sharp}$, we observe that
\begin{align*}
&\omega_A^2(T)\\
&=\omega_{A,\text{e}}^2\Big([\Re_A(T)]^{\sharp},[\Im_A(T)]^{\sharp}\Big)\\
&\leq \frac{1}{2}\Big(\left\|([\Re_A(T)]^{\sharp})^\sharp [\Re_A(T)]^{\sharp}+([\Im_A(T)]^{\sharp})^\sharp [\Im_A(T)]^{\sharp}\right\|_A\\
&+\left\|([\Re_A(T)]^{\sharp})^\sharp [\Re_A(T)]^{\sharp}-([\Im_A(T)]^{\sharp})^\sharp [\Im_A(T)]^{\sharp}\right\|_A \Big)+\omega_A\big(([\Im_A(T)]^{\sharp})^\sharp [\Re_A(T)]^{\sharp}\big).
\end{align*}
Moreover, it is not difficult to see that $([\Re_A(T)]^{\sharp})^\sharp=[\Re_A(T)]^{\sharp}$ and $([\Im_A(T)]^{\sharp})^\sharp=[\Im_A(T)]^{\sharp}$. So, we infer that
\begin{align}\label{sbou3i}
\omega_A^2(T)
&\leq \frac{1}{2}\left(\left\|([\Re_A(T)]^{\sharp})^2+([\Im_A(T)]^{\sharp})^2\right\|_A+\left\|[\Re_A(T)]^{\sharp})^2-([\Im_A(T)]^{\sharp})^2\right\|_A \right)\nonumber\\
&\quad\quad\quad\quad+\omega_A\big([\Im_A(T)]^{\sharp}[\Re_A(T)]^{\sharp}\big)\nonumber\\
&= \frac{1}{2}\left(\left\|([\Re_A(T)]^{\sharp})^2+([\Im_A(T)]^{\sharp})^2\right\|_A+\left\|[\Re_A(T)]^{\sharp})^2-([\Im_A(T)]^{\sharp})^2\right\|_A \right)\nonumber\\
&\quad\quad\quad\quad+\omega_A\big([\Re_A(T)][\Im_A(T)]\big),
\end{align}
where the last equality follows since $\omega_A(X^{\sharp}) = \omega_A(X)$ for every $X\in\mathcal{B}_{A}(\mathcal{H})$. On the other hand, by making direct calculations, it can be checked that
\begin{align*}
\left([\Re_A(T)]^{\sharp}\right)^2-\left([\Im_A(T)]^{\sharp}\right)^2
= \frac{(T^{\sharp})^2 + [(T^{\sharp})^{\sharp}]^2}{2}
= \left(\frac{T^2 + (T^{\sharp})^2}{2}\right)^{\sharp},
\end{align*}
and
\begin{align*}
\left([\Re_A(T)]^{\sharp}\right)^2+\left([\Im_A(T)]^{\sharp}\right)^2
= \frac{(T^{\sharp})^{\sharp}T^{\sharp} + T^{\sharp}(T^{\sharp})^{\sharp}}{2}
= \left(\frac{TT^{\sharp} + T^{\sharp} T}{2}\right)^{\sharp}.
\end{align*}
Hence, by taking into consideration \eqref{sbou3i} we get
\begin{align*}
\omega_A(T) \le \frac{1}{4} \left[\left\|(T^\sharp T+TT^\sharp)^\sharp\right\|_A+\left\|(T^2+(T^\sharp)^2)^\sharp\right\|_A+ \omega_A\Big((T^\sharp+T)(T-T^\sharp)\Big)\right].
\end{align*}
This proves \eqref{rr01} since ${\|X^{\sharp}\|}_A = {\|X\|}_A$ for every $X\in\mathcal{B}_{A}(\mathcal{H})$. To show the sharpness of the inequality \eqref{rr01} we choose $T = S^\sharp$ with $S$ is any $A$-selfadjoint operator on $\mathcal{H}$. So, by Lemma \ref{adma}, $S^\sharp$ is $A$-selfadjoint and $(S^\sharp)^\sharp=S^\sharp$. Thus, we deduce that
$$\omega_A\Big(\left[(S^\sharp)^\sharp+S^\sharp\right]\left[S^\sharp-(S^\sharp)^\sharp)\right]\Big)=0.$$
Further, by taking into account Lemma \ref{adma}, we get
\begin{align*}
\frac{1}{2}\sqrt{ \|(S^\sharp)^\sharp S^\sharp+S^\sharp (S^\sharp)^\sharp\|_A+\left\|(S^\sharp)^2+[(S^\sharp)^\sharp]^2\right\|_A}
& = \frac{1}{2}\sqrt{ 2\|(S^\sharp)^2\|_A+2\left\|(S^\sharp)^2\right\|_A}\\
 &=\sqrt{\|(S^\sharp)^2\|_A}\\
  &=\|S^\sharp\|_A,
\end{align*}
where the last equality follows from Lemma \ref{pos2} since $S^\sharp$ is $A$-selfadjoint. Thus, by taking into consideration \eqref{aself1}, we deduce that both sides of \eqref{rr01} become $\|S\|_A$.
\end{proof}
\begin{corollary}\label{nor1}
Let $T\in \mathcal{B}_{A}(\mathcal{H})$. Then,
\begin{align}\label{lll1}
\omega_A(T) \le \frac{1}{2}\sqrt{ \|T^\sharp T+TT^\sharp\|_A+\|T^\sharp T-TT^\sharp\|_A + \frac{1}{2}\omega_A(T^2) }.
\end{align}
Moreover, the inequality \eqref{lll1} is sharp.
\end{corollary}
\begin{proof}
By replacing $T$ and $S$ by $(T^\sharp)^\sharp$ and $T^\sharp$ respectively and using similar techniques as above we get \eqref{lll1}. To show the sharpness of the inequality \eqref{lll1} we assume that $T$ is any $A$-normal operator on $\mathcal{H}$. By \cite{feki01}, we have
\begin{equation}\label{10k}
\omega_A(T^2)=\omega_A(T)^2=\|T\|_A^2.
\end{equation}
So, it be observed that that both sides of \eqref{lll1} become $\|T\|_A$.
\end{proof}
The second inequality in Theorem \ref{mai9} can be improved as follows.
\begin{theorem}\label{mai10}
Let $T\in \mathcal{B}_A(\mathcal{H})$. Then
\begin{align}\label{sh1}
\omega_{A,\text{e}}(T,S)
 &\leq \sqrt{\max\left(\|T\|_A^2+\|S\|_A^2 \right)+\omega_A(S^\sharp T)}.
\end{align}
Moreover, the inequality \eqref{sh1} is sharp.
\end{theorem}
\begin{proof}
Let $x\in \mathcal{H}$ be such that $\|x\|_A=1$. By letting $a=x$, $b=Tx$ and $c=Sx$ in Lemma \ref{leme1} we get
\begin{align*}
|\langle Tx\mid x\rangle_A|^2+|\langle Sx\mid x\rangle_A|^2
&\leq \max\Big(\|Tx\|_A^2,\|Sx\|_A^2\Big)+|\langle Tx\mid Sx\rangle_A|\\
&\leq \max\Big(\|T\|_A^2,\|S\|_A^2\Big)+|\langle S^\sharp Tx\mid x\rangle_A|\\
&\leq \max\Big(\|T\|_A^2,\|S\|_A^2\Big)+\omega_A(S^\sharp T).
\end{align*}
Thus, by taking the supremum over all $x\in \mathbb{S}^A(0,1)$ in above inequality, we get the desired result. Now, to prove the sharpness of the inequality \eqref{sh1} we choose $T = S$, where $T$ is an $A$-selfadjoint operator. Then, by using Lemma \ref{adma}, $T^\sharp$ is $A$-selfadjoint and $(T^\sharp)^\sharp=T^\sharp$. So, we see that
\begin{align*}
\max\left(\|T^\sharp\|_A^2+\|T^\sharp\|_A^2 \right)+\omega_A\big((T^\sharp)^\sharp T^\sharp\big)
& =\|T^\sharp\|_A^2+\omega_A\big((T^\sharp)^2\big).
\end{align*}
Since $T^\sharp$ is $A$-selfadjoint, then $(T^\sharp)^2\geq_A0$. So, by \eqref{aself1}, $\omega_A\big((T^\sharp)^2\big)=\|(T^\sharp)^2\|_A$. This yields, through Lemma \ref{pos2}, that $\omega_A\big((T^\sharp)^2\big)=\|T^\sharp\|_A^2$. Thus,
$$\max\left(\|T^\sharp\|_A^2+\|T^\sharp\|_A^2 \right)+\omega_A\big((T^\sharp)^\sharp T^\sharp\big)=2\|T^\sharp\|_A^2.$$
On the other hand,
$$\omega_{A,\text{e}}^2(T^\sharp,T^\sharp)=2\omega_{A}^2(T^\sharp)=2\|T^\sharp\|_A^2.$$
\end{proof}
Now, we state the following corollary.
\begin{corollary}
Let $T\in \mathcal{B}_{A}(\mathcal{H})$. Then,
\begin{align}\label{sharpmai}
\omega_A(T) \le \frac{\sqrt{2}}{2}\sqrt{ {\|T\|_A^2 + \omega_A\big( {T^2} \big)} }.
\end{align}
The constant $\frac{\sqrt{2}}{2}$ is best possible in the sense that it cannot be replaced by a larger constant.
\end{corollary}
\begin{proof}
Let $T\in \mathcal{B}_{A}(\mathcal{H})$. By replacing $T$ and $S$ in Theorem \ref{mai10} by $T^\sharp$ and $T$ respectively,  we get
\begin{align*}
2\omega_A^2(T)
&\le \|T\|_A^2 + \omega_A\big( (T^\sharp)^2 \big)\\
 &=\|T\|_A^2 + \omega_A\big( (T^2)^\sharp \big)\\
  &=\|T\|_A^2 + \omega_A\big( T^2 \big)
\end{align*}
This proves the inequality \eqref{sharpmai}. Now, suppose that \eqref{sharpmai} holds with some constant $C > 0$. So, by choosing $T$ any $A$-normal operator (with $AT\neq 0$) and using \eqref{10k}, we easily get $\sqrt{2}C\geq 1$. This finishes the proof of the corollary.
\end{proof}
\begin{remark}
By using \eqref{refine1} together with \eqref{crucial0}, we see that
\begin{align*}
\frac{\sqrt{2}}{2}\sqrt{ {\|T\|_A^2 + \omega_A\big( {T^2} \big)} }\leq \|T\|_A.
\end{align*}
So, the inequality \eqref{sharpmai} refines the second inequality in \eqref{refine1}.
\end{remark}
The following corollary is also an immediate consequence of Theorem \ref{mai10} and its proof is similar to that given in Corollary \ref{twil} and hence omitted.
\begin{corollary}
Let $T\in \mathcal{B}_{A}(\mathcal{H})$. Then,
\begin{align}\label{5ra2020}
\omega_A(T) \le \frac{1}{2}\sqrt{ \max\left\{\|T+T^\sharp\|_A^2,\|T-T^\sharp\|_A^2 \right\}+ \omega_A\Big((T^\sharp+T)(T-T^\sharp)\Big) }.
\end{align}
Moreover, the inequality \eqref{5ra2020} is sharp.
\end{corollary}
The following corollary is an immediate consequence of Theorem \ref{mai10} and provides an upper bound for the $A$-Davis-Wielandt radius of operators in $ \mathcal{B}_A(\mathcal{H})$. The obtained result generalizes and improves \cite[Theorem 2.13]{zamanicheb2020}.
\begin{corollary}\label{th.2.11}
Let $T\in \mathcal{B}_A(\mathcal{H})$. Then,
\begin{align*}
{d\omega_A}(T) \leq \sqrt{\max\{\|T\|_A^2, \|T\|_A^4\} + \omega_A(T^\sharp T^2)}.
\end{align*}
\end{corollary}

The following lemma is useful in proving our two next results.
\begin{lemma}\label{le.2.13}
For every $a, b, c\in \mathcal{H}$, we have
\begin{align*}
|\langle a\mid b\rangle_A|^2 + |\langle a\mid c\rangle_A|^2
\leq \|a\|_A\max\left\{|\langle a\mid b\rangle_A|, |\langle a\mid c\rangle_A|\right\}
\sqrt{\|b\|_A^2+ \|c\|_A^2 + 2|\langle b\mid c\rangle_A|}.
\end{align*}
\end{lemma}
\begin{proof}
Let $a, b, c\in \mathcal{H}$. Recall from \cite[p. 132]{D.2} that
\begin{align*}
|\langle x\mid y\rangle|^2 + |\langle x\mid z\rangle|^2
\leq \|x\|\max\{|\langle x\mid y\rangle|, |\langle x\mid z\rangle|\}\Big(\|y\|^2
+ \|z\|^2 + 2|\langle y\mid z\rangle|\Big)^{\frac{1}{2}},
\end{align*}
for every $x, y, z\in \mathcal{H}$. So, by choosing $x=A^{1/2}a$, $y=A^{1/2}b$ and $z=A^{1/2}c$ in the above inequality we get the desired result.
\end{proof}
Next, we prove another upper bound for the $A$-joint numerical radius of a pair of operators.
\begin{theorem}\label{th.2.14}
Let $T\in \mathcal{B}_A(\mathcal{H})$. Then
\begin{align}\label{sh1}
\omega_{A,\text{e}}(T,S)
 &\leq \sqrt{\max\Big\{\omega_A(T),\omega_A(S)\Big\}\sqrt{\|T^\sharp T+S^\sharp S\|_A+2\omega_A(S^\sharp T)} }.
\end{align}
\end{theorem}
\begin{proof}
Let $x\in \mathbb{S}^A(0,1)$. By choosing in Lemma \ref{le.2.13} $a = x, b = Tx$ and $c = Sx$ one has
\begin{align*}
&|\langle x\mid Tx\rangle_A|^2 + |\langle x\mid Sx\rangle_A|^2\\
&\leq \|x\|_A\max\left\{|\langle x\mid Tx\rangle_A|, |\langle x\mid Sx\rangle_A|\right\}\sqrt{\|Tx\|_A^2+ \|Sx\|_A^2 + 2|\langle Tx\mid Sx\rangle_A|}\\
& \leq \max\left\{\omega_A(T),\omega_A(S)\right\} \sqrt{\big\langle \left(T^\sharp T+S^\sharp S\right)x\mid x\big\rangle_A + 2|\langle S^\sharp Tx\mid x\rangle_A|}\\
& \leq  \max\left\{\omega_A(T),\omega_A(S)\right\}\sqrt{\omega_A(T^\sharp T+S^\sharp S)+\omega_A(S^\sharp T)}\\
& =\max\left\{\omega_A(T),\omega_A(S)\right\}\sqrt{\|T^\sharp T+S^\sharp S\|_A+\omega_A(S^\sharp T)},
\end{align*}
where the last inequality follows from \eqref{aself1} since $T^\sharp T+S^\sharp S\geq_A 0$. Thus,
\begin{align*}
|\langle x\mid Tx\rangle_A|^2 + |\langle x\mid Sx\rangle_A|^2 \leq \max\left(\omega_A(T),\omega_A(S)\right)+\sqrt{\|T^\sharp T+S^\sharp S\|_A+\omega_A(S^\sharp T)},
\end{align*}
for all $x\in \mathbb{S}^A(0,1)$. Therefore, the desired result follows immediately by taking the supremum over all $x\in \mathbb{S}^A(0,1)$.
\end{proof}

\begin{corollary}
Let $T\in \mathcal{B}_{A}(\mathcal{H})$. Then,
\begin{align}
\omega_A(T) \le \frac{\sqrt{2}}{2}\sqrt{\|T\|_A\sqrt{ \|T^\sharp T+TT^\sharp\|_A+2\omega_A(T^2)} }\leq \|T\|_A.
\end{align}
\end{corollary}
\begin{proof}
Follows immediately by replacing $T$ and $S$ by $(T^\sharp)^\sharp$ and $T^\sharp$ respectively in Theorem \ref{th.2.14} and then using the second inequality in \eqref{refine1}.
\end{proof}
The following corollary in an immediate consequence of Theorem \ref{th.2.14} and generalizes \cite[Theorem 2.16]{zamanicheb2020}.
\begin{corollary}\label{improv}
Let $T\in \mathcal{B}_A(\mathcal{H})$. Then,
\begin{align*}
{d\omega_A}(T) \leq \sqrt{\max\left\{\omega_A(T), \omega_A(T^\sharp T)\right\}\sqrt{\omega_A\left[(T^\sharp T)^2 + T^\sharp T\right] + 2\omega_A(T^\sharp T^2)}}.
\end{align*}
\end{corollary}

By using Lemma \ref{le.2.13}, another upper bound for the $A$-Davis--Wielandt radius of operators in $\mathcal{B}_A(\mathcal{H})$ can be derived as follows.
\begin{theorem}\label{th.2.15}
Let $T\in \mathcal{B}_A(\mathcal{H})$. Then,
\begin{align*}
{d\omega_A}(T) \leq \sqrt{\|T\|_A\,\max\left\{\omega_A(T), \omega_A(T^\sharp T)\right\}\sqrt{1 + \|T\|_A^2 + 2\omega_A(T)}}.
\end{align*}
\end{theorem}
\begin{proof}
Let $x\in \mathbb{S}^A(0,1)$. By choosing in Lemma \ref{le.2.13} $a = Tx, b = x$ and $c = Tx$ we observe that
\begin{align*}
&|\langle Tx\mid x\rangle_A|^2 + \|Tx\|_A^4 \\
&= |\langle Tx\mid x\rangle_A|^2 + |\langle Tx\mid Tx\rangle_A|^2\\
&\leq \|Tx\|_A\max\{|\langle Tx\mid x\rangle_A|, |\langle Tx\mid Tx\rangle_A|\}\sqrt{1 + \|Tx\|_A^2 + 2|\langle x\mid Tx\rangle_A|}\\
& = \|Tx\|_A\max\{|\langle Tx\mid x\rangle_A|, |\langle T^\sharp Tx\mid x\rangle_A|\}\sqrt{1 + \|Tx\|_A^2 + 2|\langle x\mid Tx\rangle_A|}\\
& \leq \|T\|_A\,\max\{\omega_A(T), \omega_A(T^\sharp T)\}\sqrt{1 + \|T\|_A^2 + 2\omega_A(T)}.
\end{align*}
Thus
\begin{align}\label{fin}
|\langle Tx\mid x\rangle_A|^2 + \|Tx\|_A^4\leq \|T\|_A\,\max\{\omega_A(T), \omega_A(T^\sharp T)\}\sqrt{1 + \|T\|^2 + 2\omega_A(T)},
\end{align}
for all $x\in \mathbb{S}^A(0,1)$. Hence, by taking the supremum over $x\in \mathbb{S}^A(0,1)$ in \eqref{fin} we obtain the required result.
\end{proof}
 The next theorem provides an upper and lower bound of the $A$-joint numerical radius of two operators in $\mathcal{B}_A(\mathcal{H})$.
 \begin{theorem}\label{th.new}
Let $T,S\in \mathcal{B}_A(\mathcal{H})$. Then,
\begin{align*}
\frac{\sqrt{2}}{2}\max\left\{\omega_A(T+S),\omega_A(T-S)\right\}\leq\omega_{A,\text{e}}(T,S) \leq \frac{\sqrt{2}}{2}\sqrt{\omega_A^2(T+S)+\omega_A^2(T-S)}.
\end{align*}
Moreover, the constant $\frac{\sqrt{2}}{2}$ is sharp in both inequalities.
\end{theorem}
\begin{proof}
For every $x\in \mathcal{H}$, we have
\begin{align*}
(|\langle Tx\mid x \rangle_A |^2 +|\langle Sx\mid x \rangle_A  | ^2 )^{\frac{1}{2}} & \geq \frac{\sqrt{2}}{2} (|\langle Tx\mid x \rangle_A  | +|\langle Sx\mid x \rangle_A  |)\\
&\geq \frac{\sqrt{2}}{2} |\langle Tx\mid x \rangle_A  \pm \langle Sx\mid x \rangle_A  | \\
&= \frac{\sqrt{2}}{2} |\langle (T \pm S) x\mid x \rangle_A  |\,.
\end{align*}
Taking supremum over all $x\in\mathbb{S}^A(0,1)$ yields that
\begin{equation}\label{ff2020}
\omega_{A,\text{e}}(T,S) \geq \frac{\sqrt{2}}{2} \omega_A(T \pm S).
\end{equation}
This proves the first inequality in Theorem \ref{th.new}. On the other hand, for every $x\in\mathbb{S}^A(0,1)$ we have
\begin{equation}\label{tt2033}
|\langle Tx\mid x \rangle_A\pm\langle Sx\mid x \rangle_A  | ^2\leq \omega_A^2(T\pm S).
\end{equation}
So, an application of the parallelogram identity for complex numbers and \eqref{tt2033} gives
\begin{align*}
|\langle Tx\mid x \rangle_A |^2 +|\langle Sx\mid x \rangle_A  | ^2
& =\frac{1}{2}\Big(|\langle Tx\mid x \rangle_A+\langle Sx\mid x \rangle_A  | ^2+|\langle Tx\mid x \rangle_A-\langle Sx\mid x \rangle_A  | ^2 \Big) \\
 &\leq \frac{1}{2}\Big(\omega_A^2(T+ S)+\omega_A^2(T- S) \Big),
\end{align*}
for every $x\in\mathbb{S}^A(0,1)$. Taking supremum over all $x\in\mathbb{S}^A(0,1)$ yields that
\[\omega_{A,\text{e}}^2(T,S) \leq \frac{1}{2}\Big(\omega_A^2(T+ S)+\omega_A^2(T- S) \Big).\]
This shows the first inequality in Theorem \ref{th.new}. For sharpness one can obtain the same quantity $\sqrt{2} \omega_A(T)$ on both sides of the inequality by putting $T=S$.
\end{proof}
The following corollary in an immediate consequence of Theorem \ref{th.new} and \eqref{aself1}.
 \begin{corollary}
Let $T,S\in \mathcal{B}_A(\mathcal{H})$ be two $A$-selfadjoint operators. Then,
\begin{align*}
\frac{\sqrt{2}}{2}\max\left\{\|T+S\|_A,\|T-S\|_A\right\}\leq\omega_{A,\text{e}}(T,S) \leq \frac{\sqrt{2}}{2}\sqrt{\|T+S\|_A^2+\|T-S\|_A^2}.
\end{align*}
\end{corollary}
Another bounds of $\omega_{A,\text{e}}(T,S)$ can be stated as follows.
 \begin{theorem}\label{mai1000}
Let $T,S\in \mathcal{B}_A(\mathcal{H})$. Then,
\begin{align}\label{th.new2}
\frac{\sqrt{2}}{2}\sqrt{\omega_A(T^2+S^2)}\leq\omega_{A,\text{e}}(T,S) \leq \sqrt{\|T^{\sharp} T+S^{\sharp}S\|_A}.
\end{align}
\end{theorem}
\begin{proof}
Notice first that the second inequality in \eqref{th.new2} follows from \eqref{1.1}. By using \eqref{ff2020}, we observe that
\begin{align*}
2\omega_{A,\text{e}}^2(T,S)
&\geq \frac{1}{2}\Big(\omega_A^2(T+S)+\omega_A^2(T-S) \Big)\\
 &\geq \frac{1}{2}\Big(\omega_A[(T+S)^2]+\omega_A[(T-S)^2] \Big)\quad (\text{by }\,\eqref{apower})\\
  &\geq \frac{1}{2}\Big(\omega_A[(T+S)^2+(T-S)^2] \Big)\\
    &= \omega_A(T^2+S^2).
\end{align*}
This proves the first inequality in \eqref{th.new2}.
\end{proof}
The following corollary is also an immediate consequence of Theorem \ref{mai1000} and generalizes the well-known inequalities proved by F. Kittaneh in \cite[Theorem 1]{FK}. Moreover, the obtained inequalities improve the bounds in \eqref{16}.
\begin{corollary}
Let $T\in \mathcal{B}_{A}(\mathcal{H})$. Then,
  \begin{equation}\label{fffeki1}
  \frac{1}{2}\sqrt{\|T^{\sharp} T+TT^{\sharp}\|_A}\le  \omega_A(T) \le \frac{\sqrt{2}}{2}\sqrt{\|T^{\sharp} T+TT^{\sharp}\|_A}.
  \end{equation}
The inequalities in \eqref{fffeki1} are sharp.
\end{corollary}
\begin{proof}
By proceeding as in the proof of Corollary \ref{twil} we get
\begin{align*}
\frac{\sqrt{2}}{2}\sqrt{\omega_A\Big(\left([\Re_A(T)]^{\sharp}\right)^2+\left([\Im_A(T)]^{\sharp}\right)^2\Big)}\leq\omega_A(T) \leq \sqrt{\left\|\left([\Re_A(T)]^{\sharp}\right)^2+\left([\Im_A(T)]^{\sharp}\right)^2\right\|_A}.
\end{align*}
Since $\left([\Re_A(T)]^{\sharp}\right)^2+\left([\Im_A(T)]^{\sharp}\right)^2\geq_A0$, then \eqref{aself1} gives
\begin{align*}
\frac{\sqrt{2}}{2}\sqrt{\left\|\left([\Re_A(T)]^{\sharp}\right)^2+\left([\Im_A(T)]^{\sharp}\right)^2\right\|_A}\leq\omega_A(T) \leq \sqrt{\left\|\left([\Re_A(T)]^{\sharp}\right)^2+\left([\Im_A(T)]^{\sharp}\right)^2\right\|_A}.
\end{align*}
This proves the desired inequalities by following the proof of Corollary \ref{twil}.
\end{proof}
In the rest of this paper, we prove several inequalities involving the $A$-Davis-Wielandt radius and the $A$-numerical radii of operators in $\mathcal{B}_A(\mathcal{H})$.

The following lemma is useful in the proof of our next result.
\begin{lemma}\label{le.2.3}
Let $S\in \mathcal{B}_A(\mathcal{H})$. Then, for every $a\in \mathbb{S}^A(0,1)$ we have
\begin{align*}
|\langle Sa\mid a\rangle_A|^2 \leq \tfrac{1}{2} |\langle S^2a\mid a\rangle_A|
+ \tfrac{1}{4} \langle (S^\sharp S + SS^\sharp)a\mid a\rangle_A.
\end{align*}
\end{lemma}
\begin{proof}
Let $x, y, z\in \mathcal{H}$ with $\|z\|_A = 1$. We first prove that
\begin{align}\label{le.2.2}
|\langle x\mid z\rangle_A\langle z\mid y\rangle_A| \leq \tfrac{1}{2}\Big(|\langle x\mid y\rangle| + \|x\|_A\,\|y\|_A\Big).
\end{align}
Since $\|A^{1/2}z\| = 1$, then by using the well-known Buzano's inequality (\cite{buzano}), we see that
\begin{align*}
|\langle x\mid z\rangle_A\langle z\mid y\rangle_A|
& = |\langle A^{1/2}x\mid A^{1/2}z\rangle\langle A^{1/2}z\mid A^{1/2}y\rangle|\\
 &\leq \tfrac{1}{2}\Big(|\langle  A^{1/2}x\mid  A^{1/2}y\rangle| + \| A^{1/2}x\|\,\| A^{1/2}y\|\Big).
\end{align*}
This proves the desired result.

Now, let $a\in \mathbb{S}^A(0,1)$. By using the by the arithmetic-geometric mean inequality and applying \eqref{le.2.2} for $x= Sa$, $z= a$ and $y = S^\sharp a$ we infer that
\begin{align*}
|\langle Sa\mid a\rangle_A|^2
 &= |\langle Sa\mid  a\rangle_A\langle a\mid  S^\sharp a\rangle_A|\\
 & \leq \tfrac{1}{2}\Big(|\langle Sa\mid  S^\sharp a\rangle_A| + \|Sa\|_A\,\|S^\sharp a\|_A\Big)\\
 & \leq \tfrac{1}{2}|\langle Sa\mid  S^\sharp a\rangle_A| + \tfrac{1}{4}\Big(\|Sa\|^2 + \|S^\sharp a\|^2\Big)\\
 & = \tfrac{1}{2}|\langle S^2a\mid  a\rangle_A| + \tfrac{1}{4} \langle (S^\sharp S + SS^\sharp)a\mid  a\rangle_A.
\end{align*}
Hence, the proof is complete.
\end{proof}
We present now the following result.
\begin{theorem}\label{th.2.4}
Let $T\in \mathcal{B}_A(\mathcal{H})$. Then, we have
\begin{align*}
{d\omega_A}(T) \leq \frac{1}{2}\sqrt{\omega_A\left(\left(T^\sharp T + T\right)^2\right)
+\omega_A\left(\left(T^\sharp T - T\right)^2\right)
+\omega_A\Big(T^\sharp T  + 2(T^\sharp T)^2 + TT^\sharp\Big)}.
\end{align*}
\end{theorem}
\begin{proof}
Let $x\in \mathbb{S}^A(0,1)$. By applying the well-known parallelogram identity for complex numbers, we see that
\begin{align}\label{jjjjjjj}
|\langle Tx\mid x\rangle_A|^2 + \|Tx\|_A^4
& = \frac{1}{2}\left(\big|\,\|Tx\|_A^2 + \langle Tx\mid x\rangle_A \big|^2 + \big|\,\|Tx\|_A^2 - \langle Tx\mid x\rangle_A \big|^2\right)\nonumber\\
& = \frac{1}{2}\left(\big|\big\langle (T^\sharp T + T)x\mid x\big\rangle_A \big|^2 + \big|\big\langle (T^\sharp T - T)x\mid x\big\rangle_A \big|^2\right).
\end{align}
On the other hand, by applying Lemma \ref{le.2.3} we see that
\begin{align*}
&\big|\big\langle (T^\sharp T + T)x\mid x\big\rangle_A \big|^2 + \big|\big\langle (T^\sharp T - T)x\mid x\big\rangle_A \big|^2\\
& \leq \frac{1}{2}\big|\big\langle (T^\sharp T + T)^2x\mid x\big\rangle_A \big|+\frac{1}{2}\big|\big\langle(T^\sharp T - T)^2x\mid x\big\rangle_A \big|\\
&\quad  + \frac{1}{4}\big\langle \left[(T^\sharp T + T)^\sharp(T^\sharp T + T) +(T^\sharp T + T^\sharp)^\sharp(T^\sharp T + T^\sharp)\right]x\mid x\big\rangle_A\\
&\quad \quad + \frac{1}{4}\big\langle \left[(T^\sharp T - T)^\sharp(T^\sharp T - T)+(T^\sharp T - T^\sharp)^\sharp(T^\sharp T - T^\sharp)\right]x\mid x\big\rangle_A.
\end{align*}
By observing that $(T^\sharp T)^\sharp=T^\sharp T$ and making short calculations, we infer that
\begin{align*}
&\big|\big\langle (T^\sharp T + T)x\mid x\big\rangle_A \big|^2 + \big|\big\langle (T^\sharp T - T)x\mid x\big\rangle_A \big|^2\\
&\leq\frac{1}{2}\big|\big\langle(T^\sharp T + T)^2x\mid x\big\rangle_A\big| + \frac{1}{2}\big|\big\langle(T^\sharp T - T)^2x\mid x\big\rangle_A\big|\\
&\quad \quad+ \frac{1}{2}\big\langle\left[T^\sharp T + 2(T^\sharp T)^2 + TT^\sharp\right]x\mid x \big\rangle_A\\
&\leq\frac{1}{2}\left[\omega_A\left(\left(T^\sharp T + T\right)^2\right)
+\omega_A\left(\left(T^\sharp T - T\right)^2\right)
+\omega_A\Big(T^\sharp T  + 2(T^\sharp T)^2 + TT^\sharp\Big)\right].
\end{align*}
Hence, by taking into account \eqref{jjjjjjj} we obtain
\begin{align*}
&|\langle Tx\mid x\rangle_A|^2 + \|Tx\|_A^4\\
&\leq \frac{1}{4}\left[\omega_A\left(\left(T^\sharp T + T\right)^2\right)
+\omega_A\left(\left(T^\sharp T - T\right)^2\right)
+\omega_A\Big(T^\sharp T  + 2(T^\sharp T)^2 + TT^\sharp\Big)\right],
\end{align*}
for all $x\in \mathbb{S}^A(0,1)$. Finally, by taking the supremum over all $x\in \mathbb{S}^A(0,1)$ in the above inequality we get the desired result.
\end{proof}

In order to prove our next upper bound for ${d\omega_A}(\cdot)$, we need the following lemma.
\begin{lemma}\label{kz01}
Let $T\in \mathcal{B}_A(\mathcal{H})$. Then, for all $x\in \mathbb{S}^A(0,1)$ we have
$$|\langle Tx\mid x\rangle_A|^2 \leq \sqrt{\langle T^\sharp Tx\mid x\rangle_A}\sqrt{\langle TT^\sharp x\mid x\rangle_A}.$$
\end{lemma}
\begin{proof}
Let $x\in \mathbb{S}^A(0,1)$. By using the Cauchy-Schwarz inequality we see that
\begin{align*}
|\langle Tx\mid x\rangle_A|^2
& =|\langle Tx\mid x\rangle_A|\cdot|\langle Tx\mid x\rangle_A| \\
 &=|\langle Tx\mid x\rangle_A|\cdot|\langle x\mid T^{\sharp}x\rangle_A| \\
 &=|\langle A^{1/2}Tx\mid A^{1/2}x\rangle|\cdot|\langle A^{1/2}x\mid A^{1/2}T^{\sharp}x\rangle| \\
 &\leq \|Tx\|_A \|T^{\sharp}x\|_A \\
  &=\sqrt{\langle T^\sharp Tx\mid x\rangle_A}\sqrt{\langle TT^\sharp x\mid x\rangle_A}.
\end{align*}
Hence, the proof is complete.
\end{proof}
Now, we are in a position to provide the following upper bound for $d\omega_A(\cdot)$.
\begin{theorem}\label{th.2.16}
Let $T\in \mathcal{B}_A(\mathcal{H})$. Then
\begin{align*}
{d\omega_A}(T) \leq \sqrt{\frac{1}{2}\,\omega_A\Big(T^\sharp T + 2 (T^\sharp T)^2 + TT^\sharp \Big)-\frac{1}{2}\,\displaystyle{\inf_{\|x\|_A = 1}}(\|Tx\|_A - \|T^\sharp x\|_A)^2}.
\end{align*}
\end{theorem}
\begin{proof}
Notice first that $(T^\sharp T)^\sharp=T^\sharp T$. Now, let $x\in \mathbb{S}^A(0,1)$. By using Lemma \ref{kz01} and the Cauchy-Schwarz inequality we obtain
\begin{align*}
&|\langle Tx\mid x\rangle_A|^2 + \|Tx\|_A^4 \\
&= |\langle Tx\mid x\rangle_A|^2 + |\langle T^\sharp Tx\mid x\rangle_A|^2\\
& \leq \sqrt{\langle T^\sharp Tx\mid x\rangle_A}\sqrt{\langle TT^\sharp x\mid x\rangle_A}+ \sqrt{\langle (T^\sharp T)^\sharp (T^\sharp T)x\mid x\rangle_A}\sqrt{\langle (T^\sharp T)(T^\sharp T)^\sharp x\mid x\rangle}\\
& = \sqrt{\langle T^\sharp Tx\mid x\rangle_A}\sqrt{\langle TT^\sharp x\mid x\rangle_A}+ \sqrt{\left\langle (T^\sharp T)^2x\mid x\right\rangle_A}\sqrt{\left\langle (T^\sharp T)^2 x\mid x\right\rangle_A}\\
& = \frac{1}{2}\left[ \langle T^\sharp Tx\mid x\rangle_A+\langle TT^\sharp x\mid x\rangle_A -\left(\sqrt{\langle T^\sharp Tx\mid x\rangle_A}-\sqrt{\langle TT^\sharp x\mid x\rangle_A}\right)^2 \right]\\
&\quad\quad+ \left\langle (T^\sharp T)^2x\mid x\right\rangle_A\\
&=\frac{1}{2}\left[\langle T^\sharp Tx\mid x\rangle_A+\langle TT^\sharp x\mid x\rangle_A+ 2\left\langle (T^\sharp T)^2x\mid x\right\rangle_A\right]\\
&\quad\quad-\tfrac{1}{2}\left(\sqrt{\langle T^\sharp Tx\mid x\rangle_A}-\sqrt{\langle TT^\sharp x\mid x\rangle_A}\right)^2\\
&=\frac{1}{2}\left\langle\left[T^\sharp T + 2(T^\sharp T)^2 + TT^\sharp\right]x\mid x\right\rangle_A -
\frac{1}{2}\left(\|Tx\|_A - \|T^\sharp x\|_A\right)^2\\
&\leq\tfrac{1}{2}\omega_A\left[T^\sharp T + 2 (T^\sharp T)^2 + TT^\sharp \right]-\tfrac{1}{2}\displaystyle{\inf_{\|x\|_A = 1}}(\|Tx\|_A - \|T^\sharp x\|_A)^2.
\end{align*}
This gives
\begin{align*}
|\langle Tx\mid x\rangle_A|^2 + \|Tx\|_A^4
\leq\tfrac{1}{2}\omega_A\left[T^\sharp T + 2 (T^\sharp T)^2 + TT^\sharp \right]-\tfrac{1}{2}\displaystyle{\inf_{\|x\|_A = 1}}(\|Tx\|_A - \|T^\sharp x\|_A)^2,
\end{align*}
for all $x\in \mathbb{S}^A(0,1)$ which in turn shows required inequality by taking the supremum over all $x\in \mathbb{S}^A(0,1)$.
\end{proof}

The next theorem provides other bound for $d\omega_A(\cdot)$.
\begin{theorem}\label{th.2.1}
Let $T\in \mathcal{B}_A(\mathcal{H})$. Then,
\begin{equation}\label{r1}
{d\omega_A}(T) \leq \sqrt{\omega_A^2\left(T^\sharp T - T\right) + 2\|T\|_A^2\omega_A(T)}.
\end{equation}
\end{theorem}
\begin{proof}
Let $x\in \mathcal{H}$ be such that $\|x\|_A=1$. Then, by making simple calculations and using the Cauchy-Schwarz inequality, we see that
\begin{align*}
|\langle Tx\mid x\rangle_A|^2 + \|Tx\|_A^4
&= \Big|\langle Tx\mid Tx\rangle_A - \langle Tx\mid x\rangle_A\Big|^2
+ 2\Re e\Big(\langle Tx\mid Tx\rangle_A\langle Tx\mid x\rangle_A\Big)\\
&= \Big|\left\langle\left(T^\sharp T - T\right)x\mid x\right\rangle_A\Big|^2
+ 2\|Tx\|_A^2\Re e\langle Tx\mid x\rangle_A\\
& \leq \omega_A^2(T^\sharp T - T) + 2\|T\|_A^2\omega_A(T).
\end{align*}
So, we get
\begin{align}\label{sup1}
|\langle Tx\mid x\rangle_A|^2 + \|Tx\|_A^4  \leq \omega_A^2(T^\sharp T - T) + 2\|T\|_A^2\omega_A(T),
\end{align}
for all $x\in \mathbb{S}^A(0,1)$. Hence, by taking the supremum over all $x\in \mathbb{S}^A(0,1)$ in \eqref{sup1}, we get \eqref{r1} as required.
\end{proof}

To prove our next result, we need the following lemma which is quoted from the proof of \cite[Theorem 2.13.]{zamani1}.

\begin{lemma}\label{le.2.5}
Let $T\in \mathcal{B}_A(\mathcal{H})$. Then
\begin{align*}
\frac{1}{2}\|Tx\|_A\leq \sqrt{\frac{\omega_A^2(T)}{2} + \frac{\omega_A(T)}{2}\sqrt{\omega_A^2(T) - |\langle Tx\mid x\rangle_A|^2}}
\end{align*}
for any $x\in \mathbb{S}^A(0,1)$.
\end{lemma}

Now, we are ready to prove another upper bound for the $A$-Davis--Wielandt radius of operators in $\mathcal{B}_A(\mathcal{H})$.
\begin{theorem}\label{th.2.6}
Let $T\in \mathcal{B}_A(\mathcal{H})$. Then
\begin{align*}
{d\omega_A}(T)
& \leq \frac{\sqrt{2}}{2}\sqrt{\omega_A(T^2) + \frac{1}{2} \omega_A\left(T^\sharp T + TT^\sharp\right)+8\mu},
\end{align*}
where $\mu=\omega_A^2(T)\left(2\omega_A^2(T) - c_A^2(T) + 2\omega_A(T)\sqrt{\omega_A^2(T) - c_A^2(T)}\right)$.
\end{theorem}
\begin{proof}
Let $x\in \mathbb{S}^A(0,1)$. It follows, from Lemma \ref{le.2.3}, that
\begin{align}\label{id.2.6.1}
|\langle Tx\mid x\rangle_A|^2
&\leq \tfrac{1}{2}|\langle T^2x\mid x\rangle_A|+ \tfrac{1}{4}\langle\left(T^\sharp T + TT^\sharp\right)x\mid x\rangle_A\nonumber\\
&\leq \tfrac{1}{2}\omega_A(T^2) + \tfrac{1}{4}\omega_A\left(T^\sharp T + TT^\sharp\right).
\end{align}
Moreover, by using Lemma \ref{le.2.5} one has
\begin{align}\label{id.2.6.2}
\|Tx\|_A^4
&\leq  16\Big(\frac{\omega_A^2(T)}{2} + \frac{\omega_A(T)}{2}\sqrt{\omega_A^2(T) - |\langle Tx\mid x\rangle_A|^2}\Big)^2\nonumber\\
& \leq 4\Big(\omega_A^2(T) +\omega_A(T)\sqrt{\omega_A^2(T) - c_A^2(T)}\Big)^2\nonumber\\
& \leq 4\omega_A^2(T)\Big(2\omega_A^2(T) - c_A^2(T) + 2\omega_A(T)\sqrt{\omega_A^2(T) - c_A^2(T)}\Big),
\end{align}
By combining \eqref{id.2.6.1} together with \eqref{id.2.6.2}, we infer that
\begin{align*}
&|\langle Tx\mid x\rangle_A|^2 + \|Tx\|_A^4\\
& \leq 4\omega_A^2(T)\left(2\omega_A^2(T) - c_A^2(T) + 2\omega_A(T)\sqrt{\omega_A^2(T) - c_A^2(T)}\right)\\
&\quad\quad\quad+\frac{1}{2}\omega_A(T^2) + \frac{1}{4} \omega_A\left(T^\sharp T + TT^\sharp\right),
\end{align*}
for all  $x\in \mathbb{S}^A(0,1)$. Therefore, we obtain the desired inequality by taking the supremum in the above inequality over all
$x\in \mathbb{S}^A(0,1)$.
\end{proof}


\end{document}

\documentclass[12pt,reqno]{amsart}
\usepackage{etoolbox}

   \makeatletter

 \patchcmd{\@setaddresses}{\scshape\ignorespaces}{\ignorespaces}{}{} 

\appto\maketitle{%
\let\@makefnmark\relax  \let\@thefnmark\relax
\ifx\@empty\addresses\else\@footnotetext{%
  \vskip-\bigskipamount\@setaddresses}
  }
\def\enddoc@text{}
\makeatother

\makeatletter
\patchcmd\maketitle
  {\uppercasenonmath\shorttitle}
  {}
  {}{}
\patchcmd\maketitle
  {\@nx\MakeUppercase{\the\toks@}}
  {\the\toks@}
  {}
  {}{}
\patchcmd\@settitle{\uppercasenonmath\@title}{\Large}{}{}
\patchcmd\@setauthors
  {\MakeUppercase{\authors}}
  {\authors}
  {}{}
\makeatother
\usepackage{amsmath,amssymb,amsthm}
\usepackage{color}
\usepackage{url}
\usepackage{tikz-cd}
\usepackage[utf8]{inputenc}
\usepackage[T1]{fontenc}
\textheight 22.5truecm \textwidth 14.5truecm
\setlength{\oddsidemargin}{0.35in}\setlength{\evensidemargin}{0.35in}

\setlength{\topmargin}{-.5cm}
\newtheorem{theorem}{Theorem}[section]
\newtheorem{definition}{Definition}[section]
\newtheorem{definitions}{Definitions}[section]
\newtheorem{notation}{Notation}[section]
\newtheorem{corollary}{Corollary}[section]
\newtheorem{proposition}{Proposition}[section]
\newtheorem{lemma}{Lemma}[section]
\newtheorem{remark}{Remark}[section]
\newtheorem{example}{Example}[section]
\numberwithin{equation}{section}
\usepackage[colorlinks=true]{hyperref}
\hypersetup{urlcolor=blue, citecolor=red , linkcolor= blue}
\usepackage[capitalise,noabbrev,nameinlink]{cleveref}      
\begin{document}
\address{$^{[1]}$ University of Sfax, Sfax, Tunisia.}
\email{\url{kais.feki@hotmail.com}}


\date{\today}
\author[Kais Feki] {\Large{Kais Feki}$^{1}$}
\title[Inequalities for the $A$-joint numerical radius of two operators and their applications]{Inequalities for the $A$-joint numerical radius of two operators and their applications}

\maketitle